\documentclass[a4paper,12pt]{article}

\usepackage{amssymb}
\usepackage{amsmath}
\usepackage{amsthm}
\usepackage[dvips]{graphicx}
\usepackage{color}
\usepackage{float}
\usepackage{mathrsfs}
\usepackage[all,cmtip]{xy}
\usepackage{enumerate}

\def\ba{\begin{align*}}
\def\ea{\end{align*}}
\def\nn{\nonumber}
\def\ra{\rightarrow}

\def\bc{\begin{center}}
\def\ec{\end{center}}
\def\bi{\begin{itemize}}
\def\ei{\end{itemize}}
\def\bn{\begin{enumerate}}
\def\en{\end{enumerate}}
\def\bmp{\begin{minipage}}
\def\emp{\end{minipage}}

\def\h{\hspace{0.25cm}}

\def\p{\partial}
\def\real{\mathbb{R}}
\def\Jay{\mathbb{J}}
\newcommand{\tr}{\textrm{tr}}
\DeclareMathOperator{\diag}{diag}
\DeclareMathOperator{\Sym}{Sym}
\DeclareMathOperator{\ind}{ind}
\DeclareMathOperator{\linspan}{span}

\newtheorem{thm}{Theorem}
\newtheorem{prop}{Proposition}[section]

\newtheorem{lemma}[prop]{Lemma}

\newtheorem{ass}{Assumption}
\newtheorem{cnd}{Condition}

\newtheorem{defn}[prop]{Definition}

\newtheorem{rem}[prop]{Remark}

\newcommand{\executeiffilenewer}[3]{%
\ifnum\pdfstrcmp{\pdffilemoddate{#1}}%
{\pdffilemoddate{#2}}>0%
{\immediate\write18{#3}}\fi%
}

\title{On homoclinic orbits to center manifolds of elliptic-hyperbolic equilibria in Hamiltonian systems}
\author{W. Giles, J.S.W. Lamb, D. Turaev}

\begin{document}

\maketitle

\begin{abstract}
We consider a Hamiltonian system which has an elliptic-hyperbolic equilibrium with a homoclinic loop. We identify the set of orbits which are homoclinic 
to the center manifold of the equilibrium via a Lyapunov-Schmidt reduction procedure. This leads to the study of a singularity which inherits certain
structure from the Hamiltonian nature of the system. Under non-degeneracy assumptions, we classify the possible Morse indices of this singularity, permitting a local description of the set of homoclinic orbits. We also consider the case of time-reversible Hamiltonian systems.
\end{abstract}


\section{Introduction}\label{Section:Intro}

\subsection{Outline}
In this article we investigate the intersection of center-stable and center-unstable manifolds of a nonhyperbolic equilibrium of a Hamiltoian system near
a homoclinic loop. Orbits lying in this intersection converge to orbits in the center manifold in both positive and negative time. In other terminology, we investigate
a question of the structure of the set of bounded solutions which are uniformly close to a given localized (decaying to zero) solution of a Hamiltonian system of ordinary
differential equations.

We locate the intersections of the center-stable and center-unstable manifolds (corresponding to these bounded solutions) 
by deriving a real valued function whose zeros correspond to the points of intersection. This function has a critical point at the origin, 
and in order to apply basic singularity theory to analyse the zero set close to the critical point, we study the eigenvalues of the Hessian matrix at this point.
This matrix inherits structure from the Hamiltonian character of the system, meaning that the spectrum of its eigenvalues is not arbitrary. The structure of 
the spectrum is investigated using the `scattering matrix' of the linearised variational equation
along the homoclinic to the equilibrium, an approach also employed in \cite{Yagasaki2000,Yagasaki}. 
A similar object appeared originally in \cite{Lerman1991,Lerman1996}. \\

The example of lowest dimension for an elliptic-hyperbolic equilibrium is the saddle-center in a two degree of freedom system. 
Here a neighbourhood of the equilibium in the center manifold is filled with a Lyapunov family of periodic orbits, parameterised 
by the value of the Hamiltonian. Lerman \cite{Lerman1991} proved the generic existence of 4 transverse homoclinics to each periodic orbit sufficiently 
close to the equilibrium (see also Grotta-Ragazzo \cite{Grotta-Ragazzo1997},  Yagasaki \cite{Yagasaki2000}), implying the existence of 
complex dynamics in each of these energy levels. This result was generalised for systems with any number of hyperbolic degrees of 
freedom in \cite{Lerman1996}. Multi-round homoclinics are also found to emerge as the system is perturbed
\cite{Mielke1992,Koltsova1995,Champneys2000}. For higher dimensions of center manifolds, 
homoclinics to invariant tori in small perturbations of completely integrable Hamiltonian systems have been found in \cite{Koltsova2005} and \cite{Delshams2010}.

In the completely integrable case, the existence of many conserved quantities usually forces 
intersections of invariant manifolds to be of higher dimension than in the general case. 
The most commonly employed method to measure the splitting of these intersections under 
perturbation is the so-called Melnikov method (for an account in the near-integrable case, see the book \cite{Wiggins1988} 
and references therein). However, the Melnikov approach does not require such additional geometric structure, 
and can be applied in general systems. The method in this paper is also a variant of Melnikov's. Since we make 
no assumption of near-integrability, the geometry of the problem is less restricted - the Lyapunov-Schmidt approach 
to Melnikov theory employed in this paper has most in common with the papers by Gruendler \cite{Gruendler1992} who studied loops 
to hyperbolic equilibria in general systems, Palmer \cite{Palmer1984} who considered periodic forcing (see also \cite{Battelli1990}), 
and latterly Yagasaki \cite{Yagasakia}, who studied periodic perturbations of Hamiltonian systems with elliptic-hyperbolic equilibria, 
whose invariant manifolds may intersect in a degenerate manner.

In \cite{Yagasaki}, Yagasaki derived the same quadratic form studied here for the case of one hyperbolic degree of freedom, 
under additional hypotheses on the homoclinic loop, by another variant of the Melnikov method. The focus in \cite{Yagasaki} 
is on the existence of heteroclinic chains between invariant tori in the center manifold. 
The results in our paper provide less detailed information about dynamical behaviour than some of those mentioned in this 
introduction, but they may provide a first step towards a more systematic approach; the knowledge of the possible structure 
present in our reduced function at the linear level could be extended to develop normal forms for problems of this type.

The organisation of this paper is as follows. In the remainder of section \ref{Section:Intro} 
we describe the set up and our assumptions, and outline the Lyapunov-Schmidt reduction. Section \ref{Section:Weighted} establishes the necessary 
results for the reduction, and begins the study of the Hessian matrix. In section \ref{scatteringmatrix} we introduce the scattering matrix, 
and derive the formula for the Hessian matrix featuring in theorem \ref{LSreductionthm}. We also prove the first part of theorem \ref{mainthm}, which
states that the Hessian matrix cannot be positive- or negative- definite. We then prove in section \ref{Section:Nearidentity}, that any symplectic 
matrix which is sufficiently close to the identity can be realised as the scattering matrix of a system which 
satisfies our assumptions and use this result in section \ref{Section:Allindef} to demonstrate
that the Hessian matrix can have any indefinite signature, using a theorem from \cite{Mirsky1958}. We then consider in section \ref{specialcases} the case
in which the system is time-reversible, as is common in examples coming from classical mechanics.

\subsection{Problem setting}
The system is defined by the ordinary differential equations
\begin{equation} \dot{u} = X_{H}(u). \label{Hamsystem} \end{equation}
on $\real^{2n}$. The right hand side of (\ref{Hamsystem}) is the Hamiltonian vector field associated with the Hamiltonian function
$H: \real^{2n} \ra \real$, which we require to be at least $C^3$. Defining the standard symplectic matrix
\[ \Jay = \left( \begin{array}{cc}  0  & I \\ -I & 0  \end{array}\right),  \]
we can write $X_H(u) = \Jay \nabla H(u)$. We also use $\Jay$ to define the standard symplectic form
\[ \omega(\cdot,\cdot) = \left<\cdot,\Jay\cdot\right> \]
and observe that $\omega(X_H(u), \cdot) = dH(u)(\cdot)$.

We assume that the origin is an elliptic-hyperbolic equlibrium of system (\ref{Hamsystem}), that is;
\begin{ass}\label{specass}
The spectrum of the linearisation $DX_H(0)$ consists of $2l$ distinct eigenvalues with zero real part, 
$\pm i\omega_j$, $j \in \{1,...,l\}$, and $2(n - l)$ eigenvalues $\lambda_i$,  whose real parts are bounded away from 
zero; $0<\alpha<|{\rm Re} \lambda_i|$ $i\in\{1,..,2(n-l)\}$.
\end{ass}
The equilibrium possesses $(n-l)$-dimensional stable and unstable manifolds $W^s$ and $W^u$, which are assumed 
to intersect along a homoclinic loop $\gamma(t)$, namely;
\begin{ass}\label{homoass}
There exists an orbit $\Gamma = \{ \gamma(t) : t \in \real\}$ such that $ \Gamma \subset W^s\cap W^u$.
\end{ass}

We denote by $E^{u}(DX_H(0))$, $E^{s}(DX_H (0))$ the unstable and stable eigenspaces of the linearisation at the origin. The centre subspace, 
corresponding to the purely imaginary eigenvalues, which is symplectic, will be denoted $E^{c}(DX_H (0))$, or simply $E^c$ when 
the context is clear. Under these assumptions, the equilibrium possesses a $2l$-dimensional center manifold, which is symplectic. The center manifold may not 
be unique, but any center manifold will be tangent at the origin to $E^c(DX_H(0))$, yielding the same linearisation, and the same result in our context. 
The restriction of $H$ to the center manifold defines a Hamiltonian system with $l$ degrees of freedom and an elliptic critical point at the origin. Writing 
the tangent space at the equilibrium according to the symplectic splitting (see  \cite{Mielke1991});
\[ \real^{2n} = E^c \oplus (E^u \oplus E^s) \]
we have
\[ \Jay = \begin{pmatrix} \Jay_1 & 0 \\ 0 & \Jay_2 \end{pmatrix}. \]
Choosing a symplectic basis on $E^c$ such that  $\Jay_1 = \left( \begin{smallmatrix} 0 & I \\ -I & 0 \end{smallmatrix} \right)$, since $\Jay_1 D^2H(0)|_{E^c}$ has 
distinct purely imaginary eigenvalues, we can (and do) make a symplectic change of coordinates in $E^c$ which brings $D^2H(0)|_{E^c}$ to the form
\[ D^2 H(0)|_{E^{c}} =\left( \begin{smallmatrix} \omega_1 & & & & & \\
                                                                                               &  \ddots & & & & \\
                                                                                                 & & \omega_l & & & \\
                                                                                           & & & \omega_1 & & \\
                                                                                           & & & & \ddots &   \\
                                                                                           & & & & & \omega_l \end{smallmatrix} \right) 
                                                                                           =: \diag (\omega_1,...,\omega_l,\omega_1,...\omega_l). \]



The origin also possesses $(n+l)$-dimensional center-stable and center-unstable manifolds $W^{cs}$ and $W^{cu}$. The orbits we seek, which converge to the center manifold in forward and backward time, are contained in the intersection $W^{cs}\cap W^{cu}$.We make the following assumption on the invariant manifolds;
\begin{ass} $\dim( T_{\gamma(0)}W^{cu} \cap T_{\gamma(0)}W^s) = \dim(T_{\gamma(0)}W^{cs} \cap T_{\gamma(0)}W^u) = 1$. \label{transass} \end{ass}
Of course, the existence of the homoclinic $\gamma(t)$ which is contained in the intersection of $W^u$ and $W^s$ implies that the dimension of the intersection in
assumption \ref{transass} is at least one, so this assumption means that this dimension is minimal: there is no further degeneracy leading to a higher-dimensional intersection.

\subsection{Statement of results}

\begin{thm} \label{LSreductionthm}
 Under the assumptions \ref{specass}, \ref {homoass} and \ref{transass}, homoclinic orbits to the center manifold of the origin correspond to zeros of a function
$\mathfrak{g} : \real^{2l} \ra \real$, which has the property that $\nabla \mathfrak{g}(0) = 0$, and its Hessian is given by
\[  D^2\mathfrak{g} =  \sigma^T D^2 H(0)|_{E^c} \sigma - D^2 H(0)|_{E^c}\]
where $\sigma$ is the symplectic scattering matrix\footnote{see the definition of the scattering
matrix in section \ref{scatteringmatrix}} determined by the flow linearized about the homoclinic loop $\Gamma$.
\end{thm}

The zeros of the function $\mathfrak{g}$ correspond to intersections of $W^{cs}$ and $W^{cu}$. These manifolds are foliated by the strong-stable and, resp.,
strong-unstable leaves of the points in $W^c$: if a forward orbit of a point $M\in W^c$ stays in a small neighbourhood of the equilibrium at the origin, then
its strong-stable leaf $l^{ss}(M)$ consists of all points whose forward orbits tend to the forward orbit of $M$ exponentially with a rate at least $e^{-\alpha t}$,
the same for the strong-unstable leaf $l^{uu}(M)$ and backward orbits.
We prove in Theorem \ref{LSreductionthm} that $\mathfrak{g}(M)=0$ if and only if there exists a point $\bar M\in W^c$ such that $l^{uu}(M)$ has a point of intersection
with $l^{ss}(\bar M)$, and the orbit of this intersection point is close to the homoclinic loop $\Gamma$ when it goes from a small neighbourhood of $M$ to a small neighbourhood of $\bar M$. This orbit is homoclinic to $W^c$ (and corresponds to a solution of system (\ref{Hamsystem}) which is bounded and uniformly close to $\gamma(t)$
 for all $t$) if both the backward orbit of $M$ and the forward orbit of $\bar M$ are bounded and stay close to the origin.

Note that an equivalent quadratic form is derived in \cite{Yagasaki} in the case of one hyperbolic degree of freedom (i.e $n = l+1$ in our notation), under additional, and quite strong assumption that the homoclinic loop is contained in a normally elliptic invariant manifold. There, under the extra (mild) hypotheses necessary for KAM type results to yield the existence of a family of invariant tori in the center manifold, Yagasaki proves that when expressed in polar coordinates, a zero of the quadratic form at which the radial derivatives are nonzero corresponds to a transversal intersection of invariant manifolds of two invariant KAM-tori. The existence of chains of heteroclinic orbits shadowed by real `diffusing' orbits is then shown for some examples. It would seem that the methods from the current paper combined with those from \cite{Yagasaki} allow one to prove the existence of chains and accompanying diffusion behaviour in a much larger class of far-from-integrable Hamiltonian systems. \\

By Morse lemma, if the quadratic part of $\mathfrak{g}$ is non-degenerate, the structure of the zero set of $\mathfrak{g}$ is determined by the
signature of $D^2\mathfrak{g}$. The following theorem, the main result of this paper, describes the possible signatures of $D^2\mathfrak{g}$
compatible with the Hamiltonian structure of the system.
\begin{thm} \label{mainthm}
Under assumptions  \ref{specass}, \ref {homoass}, \ref{transass},
\begin{enumerate}[(i)]
\item $D^2\mathfrak{g}$ can be neither positive nor negative definite.
\item All indefinite signatures for $D^2\mathfrak{g}$ can be realised by systems satisfying the assumptions. 
Furthermore, they can be realised in systems which are a small perturbation of a completely integrable system.
\end{enumerate}
\end{thm}

The first part of the theorem says that as long as the critical point of $\mathfrak{g}$ is Morse, the homoclinic $\gamma(t)$ is never 
an isolated intersection point of the center stable and center unstable manifolds - a situation which in the general (non-Hamiltonian) case 
could arise. In the case $l=1$, we find agreement with a result from \cite{Lerman1996}; the existence of one positive and 
one negative eigenvalue leads to a degenerate hyperpola (a `cross') for the zero set of $\mathfrak{g}$, which intersects each sufficiently small periodic 
orbit surrounding the origin in 4 places, leading to 4 homoclinics. The rest of the theorem says that in general there is no 
further restriction on the singularity.

In section \ref{specialcases} we will consider also the case in which the vector field is reversible;

\begin{ass} \label{Reversibility} Letting $R$ be a linear involution which acts antisymplectically, that is $R^2 = I$ and $R\Jay = -\Jay R$,
\begin{enumerate}[(i)]
\item $X_H$ is R-reversible: $ X_H(Ru) = -RX_H(u).$
\item The homoclinic $\gamma(t)$ is $R-symmetric$: writing $\Gamma = \{ \gamma(t) : t \in \real \}$, we have $ R\Gamma  = \Gamma. $
\end{enumerate}
\end{ass}

In this case we find;

\begin{thm} \label{specialcasesthm}
Under assumptions \ref{specass}, \ref {homoass}, \ref{transass}, \ref{Reversibility}, the signature of $D^2\mathfrak{g}$ is $(l,l)$.
\end{thm}

\subsection{The Lyapunov-Schmidt reduction} \label{Subsection:Reductiontoasingularity}

Returning now to system (\ref{Hamsystem}), that is,
\begin{equation} \dot{u} = X_{H}(u) \nonumber \end{equation}
with the homoclinic orbit $\gamma(t)$, we seek homoclinic orbits $\tilde{\gamma}(t)$ as perturbations of $\gamma(t)$, by first writing
\[ \tilde{\gamma}(t) = \gamma(t) + x(t). \]
Substituting this into (\ref{Hamsystem}) and rearranging for $x(t)$ brings us to the equation
\begin{equation}\label{Hamvarsystem} \dot{x}(t) = X_{H}(\gamma(t) + x(t)) - X_{H}(\gamma(t)) \end{equation}
We then define an operator $F$ by
\[ F(x) := \dot{x}(t) - X_{H}(\gamma(t) + x(t)) - X_{H}(\gamma(t)) \]
so that zeros of $F$ correspond to solutions of (\ref{Hamvarsystem}).

For convenience, we will modify $X_H$ outside a small neighbourhood of the homoclinic loop $\Gamma$ so that it will be identically zero
outside some larger (still small) neighbourhood of $\Gamma$ (this is achieved by multiplying the Hamiltonian to a smooth cut-off function, equal
to $1$ everywhere near $\Gamma$ and zero everywhere outside a small neighbourhood of $\Gamma$). Then, all zeros of (the modified)
operator $F$ which are uniformly close to zero will correspond to solutions of the original system which are uniformly close to $\gamma(t)$.

By choosing an appropriate domain $\mathcal{X}$ and target space $\mathcal{Y}$ for $F$,
we can search for solutions $x(t)$ which satisfy prescribed conditions on their asymptotic behaviour, which corresponds
to finding homoclinic solutions with desired features. Clearly, $F(0) = 0$. Taking a Frechet derivative of $F$ at $0$ leads us to the operator
\[ DF(0)x(t) := Lx(t) = \dot{x}(t) - DX_{H}(\gamma(t))x(t) \]
so that zeros of $L$ are solutions of the \emph{variational equation}
\begin{equation}\dot{x}(t) = DX_{H}(\gamma(t)) x(t). \label{Hamvareqn} \end{equation}
Note that one solution (which, since $\gamma(t)$ lies in the intersection of the stable and unstable manifolds of the equilibrium, decays
exponentially fast in both forward and backward time) of (\ref{Hamvareqn}) is given by $\dot{\gamma}(t)$.
A crucial point, discussed in more detail in the following section, is that $L$ is a \emph{Fredholm} operator.
This means by definition that $\ker(L) \subset \mathcal{X}$ is finite-dimensional, and the range $\mathcal{R}(L) \subset \mathcal{Y}$ is of finite codimension. The index of $L$ is then the integer ind$(L) = \dim \ker(L) - \mbox{codim}(\mathcal{R}(L))$. This will allow us to perform a
Lyapunov-Schmidt reduction of the map $F$ at zero. The procedure is as follows; we decompose $\mathcal{X}$ and $\mathcal{Y}$  in the following way
\begin{align*} \mathcal{X} & = \ker(L) \oplus \mathcal{M} \\
               \mathcal{Y} & = \mathcal{N} \oplus \mathcal{R}(L) \end{align*}
and now look for solutions of the following equivalent system, where the variable $x = k + w$ is split according to the decomposition of $\mathcal{X}$ and $P$ is the projection onto $\mathcal{R}(L)$ in $\mathcal{Y}$ with $\ker(P) = \mathcal{N}$;
\begin{equation} \begin{cases} \h\h PF(k + w) & = 0 \\
           (I - P)F(k + w) & = 0. \end{cases} \label{Splitsystem} \end{equation}
The advantage of this construction is that the derivative in the first component of the system with repect to $w$, $D_w PF(v)|_{(0)}$, is invertible, and so we can use the implicit function theorem to locally solve this first equation. This allows us to write $v \in \mathcal{X}$ as $k + w(k)$, where $w:\ker(L) \ra \mathcal{M}$ is such that
\[PF(k + \sigma) = 0 \Leftrightarrow \sigma = w(k) \]
in a neighbourhood of $x=0$. We note also that $D_k w(0) = 0$; differentiating the top component of (\ref{Splitsystem}) with respect to $k$ at zero leads to
\[ L(D_k w(0)) = 0, \]
and since $D_k w(0) \in \mathcal{M}$, we can invert $L$, yielding $D_k w(0) = 0$.
We are then only required to find zeros of the map defined by
\[ (I - P)F(k + w(k)):\ker(L) \ra \mathcal{N}, \]
which we denote by $\mathcal{G}(k)$. Zeros of $\mathcal{G}$ then correspond to zeros of the full system. Furthermore, $\mathcal{G}$ has a critical point at the origin:
\[ D\mathcal{G}(0) = (I-P)DF(k)(D_k w(k))|_{k=0} = 0. \]
So, as long as this singularity is nondegenerate (i.e. the Hessian matrix is invertible) we can locally describe the zero set of $\mathcal{G}$ by classifying the critical point at the origin and appealing to the Morse lemma, which gives us a normal form for the quadratic part of $\mathcal{G}$.


\section{Weighted function spaces and Fredholm properties} \label{Section:Weighted}

For $\beta \in \real$, we define the Banach space
\[ C^1_{\beta}(\real, \real^{2n}) = \{ x: \real \ra \real^{2n} \mbox{ with } \sup_{t\in\real}\| e^{\beta |t|} x(t) \| < \infty, \: \sup_{t\in\real}\| e^{\beta |t|} \dot{x}(t) \| < \infty \}. \]
We will require the following result;
\begin{lemma}
There exists a $\beta \in (0, \alpha)$ such that a solution $x(t)$ of the equation $F(x) = 0$ gives rise to an orbit $\tilde{\gamma}(t) = \gamma(t) + x(t)$ which remains in a tubular neighbourhood of $\gamma(t)$ and is homoclinic to the centre manifold of the origin  if and only
if $x(t) \in C^1_{-\beta}(\real, \real^{2n})$, and $x(t)$ is uniformly small for all $t$.
\end{lemma}
\begin{proof}
If $ \tilde{\gamma}(t) \in W^{cu}(0)\cap W^{cs}(0)$, then it approaches an orbit $\eta(t)$ in the center manifold;
\[ \| \tilde{\gamma}(t) - \eta(t) \| \ra 0 \mbox{ as } t \ra \infty \]
but since $\gamma(t) \ra 0 $ exponentially fast, we have
\begin{equation} \| x(t) - \eta(t) \|  \ra 0 \mbox{ as } t \ra \infty \label{xtendstoeta} \end{equation}
Since $\eta(t)$ is contained in the center manifold, $\eta \in C^1_{-\beta}(\real, \real^{2n})$ for any $\beta \in (0, \alpha)$, and hence we can use (\ref{xtendstoeta}) to conclude that $x \in C^1_{-\beta}(\real,\real^{2n})$. Moreover, if $\tilde{\gamma}(t)$ lies in a small tubular neighbourhood of $\gamma(t)$ then the norm of $x(t)$ is necessarily small.

Conversely, assume that $x(t)$ is uniformly small in norm, and $F(x) = 0$. Then $\tilde{\gamma}(t)=x(t)+\gamma(t)$ defines a trajectory of system which stays in a small tubular neighbourhood of $\gamma(t)$. In particular it stays in a small neighbourhood of zero for all sufficiently large values of $|t|$. Therefore, just
by the definition of the center-stable and center-unstable manifolds $\tilde\gamma(t)$ stays in $W^{cs}$ for all large $t>0$, hence it must tend to
a bounded orbit in  $W^c$ as $t\to+\infty$, and, as $t\to-\infty$, it stays in $W^{cu}$, which implies that it tends to a bounded orbit in $W^c$ as $t\to-\infty$
as well (see e.g. \cite{book} for more detail).
\end{proof}

This result justifies the use of an exponentially weighted norm on the domain of $F$ to capture all of the solutions which do not grow faster than a given exponential factor.
 Letting $\phi(t) \in C^1(\real,\real)$ be such that
\[ \begin{cases} \phi(t) = |t| \mbox{ for } t \in (-\infty, -1] \cup [1, \infty) \\
                           \sup_{t \in [-1,1]} | \phi(t) - |t| | << 1    \\
                           \phi(t) > 0 \mbox{ for } t\in \real,  \end{cases} \]
we now consider the weighted inner product
\[ <u,v>_{\delta} \:\: = \int_{\real} e^{-2\delta\phi(t)}<u(t),v(t)>dt \]
which is defined for any $u, v \in C^1_{-\beta}(\real, \real^{2n})$ with $0<\beta < \delta$. We are hence free to choose $\beta$, $\delta$ satisfying the following condition.
\begin{cnd}
The constants satisfy $0 < \beta < \delta < \alpha$, and $\delta - \alpha < \beta - \delta$, where $\alpha$ is as defined in assumption \ref{specass}. \label{Constantsass}
\end{cnd}

We calculate an expression for the adjoint $L^*$ with respect to the weighted inner product as follows;
\begin{align*}
 \left< L u, v \right>_{\delta} =  & \int_{\real}  e^{-2\delta \phi(t)} \left< \dot{u}(t) - DX_H(\gamma(t))u(t), v(t) \right> dt \\
                                                                        =  & \int_{\real}  \left< \dot{u}(t), e^{-2\delta \phi(t)}v(t) \right>  - e^{-2\delta \phi(t)}\left< DX_H(\gamma(t))u(t), v(t) \right> dt
\end{align*}
\begin{align*}
                                                                        =  & \int_{\real} - \left< u(t), \frac{d}{dt}(e^{-2\delta \phi(t)}v(t)) \right>  - e^{-2\delta \phi(t)}\bigg< u(t), DX_H(\gamma(t))^{*} v(t) \bigg> dt \\
                                                                        = & \int_{\real} - e^{-2\delta \phi(t)}\left< u(t), \frac{d}{dt}v(t) -2\delta \dot{\phi}(t)v(t) \right>  - e^{-2\delta \phi(t)}\bigg< u(t), DX_H(\gamma(t))^{*} v(t) \bigg> dt,
\end{align*}

We conclude from this line that
\[ L^* = -\frac{d}{dt} + 2\delta \dot{\phi}(t) - DX_H(\gamma(t))^{*} \]
we refer to $L^*u = 0$ as the \emph{adjoint variational equation}.

\begin{lemma}
$DF(0,0) := L : C^{1}_{-\beta}(\real,\real^{2n}) \ra C^{0}_{-\beta}(\real,\real^{2n})$ is a Fredholm operator of index $2l$. Furthermore,
$y(t) \in \mathcal{R}(L)$ if and only if
\[ \int_{\real} e^{-2\delta \phi(t)} \left< y(t), \psi(t) \right>dt = 0, \mbox{ for every $\psi \in C^1_{-\beta}$ solving } L^* \psi = 0. \]
 \label{Fredholmlemma}
\end{lemma}

\noindent To prove lemma \ref{Fredholmlemma} we will make use of a conjugacy beetween $L$ and a `shifted' version of $L$ on a differently weighted function space. We observe that $L = DF(0,0) : C^1_{-\beta} \ra C^0_{-\beta}$ is conjugate to the shifted operator $ L_{\delta} : C^1_{\delta - \beta}  \ra C^0_{\delta - \beta}$ given by
\[  L_{\delta}u(t)  = \frac{du}{dt} - \delta \dot{\phi}(t) u(t) - DX_{\tilde{H}}(\gamma(t))u(t), \]
The conjugacy is given by the isomorphism $v(t) \mapsto e^{-\delta \phi(t)}v(t)$ which maps from $C^1_{-\beta}$ into $ C^1_{\delta - \beta}$, which is endowed with the unweighted inner product. The utility of this conjugacy stems from the fact that the limits
\begin{equation} \lim_{t\ra \pm \infty} (\delta \dot{\phi}(t) I + DX_{H}(\gamma(t)) ) \label{Limitmatrices} \end{equation}
are now hyperbolic, since the imaginary eigenvalues of $DX_{H}(0)$ are now shifted, to the right of the imaginary axis in negative time and to the left in positive time. We will make use of the following theorem:\\

\noindent{\bf Palmer Theorem} \cite{Palmer1984} {\em
Let $A(t)$ be an $n \times n$ matrix function bounded and continuous on $\real$ and such that
\[ \lim_{t \ra -\infty} A(t) = A_{-\infty}, \mbox{    } \lim_{t\ra \infty} A(t) = A_{\infty} \]
exist and are hyperbolic. Then
\begin{align*}  B : C^1( & \real, \real^{2n}) \ra C^0(\real, \real^{2n}) \\
                                   Bx & = \dot{x}(t) - A(t)x(t) \end{align*}
is Fredholm, and $y \in \mathcal{R}(B)$ if and only if
\[ \int_{\real} \left< y(t), \psi(t) \right>dt = 0, \mbox{ for every bounded $\psi$ solving } \dot{\psi}(t) = -A^*(t)\psi(t). \]
Furthermore, if $A_{-\infty}$, $A_{\infty}$ have $a_-$ and $a_+$ unstable eigenvalues respectively, then
\[ \ind(L) = a_- - a_+. \]}\\

\begin{proof}[Proof of lemma \ref{Fredholmlemma}]
We first consider our shifted operator defined on the larger function space $C^1(\real, \real^{2n})$ of bounded continuous functions, as in the statement of Palmer theorem. Call this operator $\hat{L}_\delta$. Applying Palmer theorem to $\hat{L}_\delta$  tells us that the index of $\hat{L}_\delta = 2l$. Firstly this means that $\ker(\hat{L}_\delta) < \infty$. This remains true for $L_\delta$, since $\ker(\hat{L}_\delta) = \ker(L_{\delta})$: any bounded solutions decay at a rate of at least $e^{\delta - \alpha}$ in negative time and $e^{-(\alpha + \delta)}$ in positive time (as can be seen by looking at the spectrum of the limit matrices in (\ref{Limitmatrices})), and so, in particular, faster than $e^{\beta - \delta}$ in both time directions, as a consequence of condition \ref{Constantsass}. Hence, these solutions lie in $C^1_{\delta - \beta}$.

The application of Palmer theorem also gives $\mathcal{R}(\hat{L}_{\delta}) = \ker(\hat{L}^*_{\delta})^{\perp}$. We find that $\ker(\hat{L}^*_{\delta}) = \ker(L^*_{\delta})$ for the same reasons as in the previous paragraph, and so
\[ \ker(L^*_{\delta})^{\perp} = \mathcal{R}(\hat{L}_{\delta})\cap C^0_{\delta - \beta} = \mathcal{R}(L_{\delta}) \]
is it clear from these considerations that $\ind(\hat{L}_\delta) = \ind(L_\delta)$.

Finally, applying the inverse of the conjugacy brings us back to the original operator $L$, preserving the required properties.
\end{proof}
\begin{figure}[h]
\centering
\def\svgwidth{80mm}
\caption{Eigenvalues cross the imaginary axis from right to left, as time progresses through $\real$, inducing a positive Fredholm index.}
\end{figure}

We note that assumption \ref{transass} implies that $\dot{\gamma}(t)$ is the only solution (up to a scalar multiple) of the variational equation which decays at an exponential rate (in fact, $\dot{\gamma}(t) \in C^1_{\alpha}(\real, \real^{2n})$). This also implies that the only (again, up to a scalar multiple) exponentially decaying solution of the adjoint variational equation (with respect to the unweighted inner product) is given by $\Jay\dot{\gamma}(t) = \nabla H(\gamma(t))$.

\begin{lemma}
$ \ker(L^*)  = \linspan \{ e^{2\delta \phi(t) }\nabla H(\gamma(t)) \}$.
\end{lemma}
\begin{proof}
If $L^*u = 0$ with $u(t)\in C^1_{-\beta}(\real, \real^{2n})$, then
\begin{align*} 0 = L^*u & = e^{-2\delta\phi(t)}L^*u \\
   & =  - e^{-2\delta \phi(t)}\frac{du}{dt} + 2\delta\dot{\phi}(t) e^{-2\delta\phi(t)}u(t)  - DX_{\tilde{H}}(\gamma(t),0)^*e^{-2\delta\phi(t)} u(t)    \\
& = -\frac{d}{dt}(e^{-2\delta\phi(t)}u(t))  - DX_{\tilde{H}}(\gamma(t),0)^*e^{-2\delta\phi(t)} u(t) \end{align*}
The expression on the right hand side here is the adjoint variational equation with respect to the unweighted inner product. Now, $e^{-2\delta\phi(t)}u(t)$ is an exponentially decaying solution of the unweighted adjoint variational equation, and hence $e^{-2\delta\phi(t)}u(t) \in \linspan\{ \nabla H(\gamma(t)) \}$, meaning that $u(t) \in \linspan \{ e^{2\delta \phi(t) }\nabla H(\gamma(t)) \}$.

Similarly, if $ v(t) \in \linspan\{ \nabla H (\gamma(t))\}$, then v(t) solves
\[ -\frac{d}{dt}v(t) - DX_{\tilde{H}}(\gamma(t),0)^* v(t) = 0 \]
while $e^{2\delta \phi(t)}v(t) \in C^1_{-\beta}(\real, \real^{2n})$ and
\begin{align*} -\frac{d}{dt}(e^{2\delta\phi(t)v(t)}) & +  2\delta \dot{\phi}(t) e^{2\delta\phi(t)}v(t)   - e^{2\delta\phi(t)}DX_{\tilde{H}}(\gamma(t),0)^*v(t)  = \\
\vspace{-2mm}& = -2\delta \dot{\phi}(t) e^{2\delta\phi(t)}v(t) - e^{2\delta\phi(t)}\frac{d}{dt}v(t) + 2\delta \dot{\phi}(t)e^{2\delta\phi(t)}v(t)  \\ & \hspace{2cm} - e^{2\delta\phi(t)}DX_{\tilde{H}}(\gamma(t),0)^*v(t) \\
                         \vspace{-2mm}         & = e^{2\delta\phi(t)}(L^*(v(t))) = 0 \end{align*}
\end{proof}
Hence $\mathcal{R}(L)^{\perp}$ is one-dimensional, and so $\dim(\ker(L)) = 1 + $ind$(L) = 2l + 1$.

As a check, we observe that if $\psi(t) \in C^1_{-\beta}(\real,\real^{2n})$ is a solution of the adjoint variational equation, and   $f \in \mathcal{R}(L)$, that is, $f(t) = \dot{x}(t) - DX_{H}(q(t),0)x(t)$ for some $x(t) \in C^1_{-\beta}(\real,\real^{2n}) $ then
\begin{align*}
\left<\psi(t), f(t) \right>_{\delta} = & \int_{\real} e^{-2\delta \phi(t)} \bigg< \psi(t), \dot{x}(t) - DX_H(t)(\gamma(t),0)x(t)\bigg> dt \\
                                                       = & \int_{\real} e^{-2\delta \phi(t)}\bigg< \psi (t), \dot{x}(t) \bigg> - \bigg< DX_H(\gamma(t),0)^* \psi(t), x(t) \bigg> dt  \\
                                                       = & \int_{\real} e^{-2\delta \phi(t)}\bigg< \psi (t), \dot{x}(t) \bigg> + \bigg<\dot{\psi}(t) - 2\delta\dot{\phi}(t)\psi(t), x(t) \bigg> dt \\
                                                       = & \int_{\real} e^{-2\delta \phi(t)}\left( \frac{d}{dt} \bigg< \psi(t),x(t) \bigg> - 2\delta\dot{\phi}(t)\bigg<\psi(t), x(t)\bigg> \right) dt \\
                                                       = & \int_{\real} \frac{d}{dt}\left( e^{-2\delta \phi(t)}\bigg<\psi(t), x(t)\bigg>\right)dt \\
                                                       = & \bigg[ e^{-2\delta \phi(t)}\bigg<\psi(t), x(t)\bigg> \bigg]^{\infty}_{-\infty} = 0
\end{align*}

When we construct a reduced map by Lyapunov-Schmidt reduction, we project onto $\ker(L^*)$ by taking the weighted inner product with this unique exponentially decaying solution. The exponential factors in the weight and the solution will then cancel, leaving us with an expression which involves an unweighted inner product.

The results from this section facilitate a Lyapunov-Schmidt reduction of the map $F$ at zero according to a decomposition of the following form;
\[    C^{1}_{-\beta}(\real, \real^{2n}) = \ker(L) \oplus \mathcal{M} \]
\[    C^{0}_{-\beta}(\real, \real^{2n}) = \ker(L^*) \oplus \mathcal{R}(L) \]

Performing the reduction as described in section \ref{Subsection:Reductiontoasingularity} leads to the reduced map
\[ \mathcal{G}(k): = (I - P)F(k + w(k)):\ker(L) \ra \ker(L^*) \]
so $\mathcal{G}$ maps from a $(2l+1)$-dimensional space into a $1$-dimensional space, as a consequence of the
positive Fredholm index of $L$, and $\mathcal{G}$ has a critical point at the origin. This proves the first part of theorem \ref{LSreductionthm}.

\subsection{The Hessian Matrix} \label{Section:Singularitytheory}
We now study this critical point of the reduced map $\mathcal{G}(k)$ by investigating the Hessian matrix. For the calculations, we now let  $k_i$, $i \in \{1, ... , 2l +1\}$ be a chosen basis of $\ker (L)$, with $k_1 =  \dot{\gamma}(t)$, and we write $\mathfrak{g}(\beta_1,...,\beta_{2l+1}) :=\mathcal{G}(\beta_1 k_1, ... , \beta_{2l+1} k_{2l+1})$, so that

\begin{align*}
\mathfrak{g}(\beta) =  \int_{\real} e^{-2\delta \phi(t)}& \Bigl< \Bigr. e^{2\delta \phi(t)}\nabla H(\gamma(t)),  \dot{\gamma}(t) + \Sigma_i \beta_i \dot{k_i}(t) \\
& + \dot{w}(\beta)(t) - X_H(\gamma(t) +  \Sigma_i \beta_i k_i(t) + w(\beta)(t)) \Bigl. \Bigr> dt   \end{align*}

The following lemma provides a formula for the derivatives of $\mathfrak{g}(0)$. The proof is the same in essence as the one in \cite{Gruendler1992} (theorem 5), in which a homoclinic orbit to a hyperbolic equilibrium is studied. We include the proof here for completeness.
\begin{lemma}
\begin{subequations}
\begin{eqnarray}
\frac{\p \mathfrak{g}}{\p \beta_i}(0) & = & 0  \label{reduced derivs 1} \\
 \frac{\p^2 \mathfrak{g}}{\p \beta_i \p \beta_j}(0) & = & \int_{\real} \left< \dot{\gamma}(t),  D_{x}^3 H(\gamma(t))(k_i(t),k_j(t)\right>dt     \label{reduced derivs 2}
\end{eqnarray}
\label{reduced derivs}
\end{subequations}
\end{lemma}

\begin{proof}
The first equation simply states that the reduced map has a singularity at the origin, which is true for any map produced in this way via the Lyapunov-Schmidt reduction, as discussed in section \ref{Subsection:Reductiontoasingularity}. As for the second, differentiating $\mathfrak{g}$ twice and evaluating at $\beta = 0$ gives:
\begin{align*}
\frac{\p^2 \mathfrak{g}}{\p \beta_i \p \beta_j} = & \int_{\real} \left< \nabla H(\gamma(t)), \frac{\p^2 \dot{w}(0)}{\p \beta_i \p \beta_j} - DX_H(\gamma(t)) \frac{\p^2 w(0)}{\p \beta_i \p \beta_j} \right> dt \\
& - \int_{\real} \left< \nabla H(\gamma(t)), D^2X_H(\gamma(t))(k_i(t),k_j(t)) \right> dt
\end{align*}
and the first term is zero for each $(i,j)$, since $\frac{\p^2 \dot{w}(0)}{\p \beta_i \p \beta_j} - DX_H(\gamma(t)) \frac{\p^2 w(0)}{\p \beta_i \p \beta_j}$ lies in the range of $L$. The final step is to recall that $X_H$ can be written as $-\Jay \nabla H$, and that $\dot{\gamma}(t) = -\Jay \nabla H(\gamma(t))$. Applying the isometry $\Jay$ in both sides of the inner product and using these facts yields (\ref{reduced derivs 2}).
\end{proof}

In fact, we can restrict our attention to finding zeros of $\mathfrak{g}$ with its first argument (the coefficient of $\dot{\gamma}(t)$) fixed at zero. Considering the direct sum decomposition
\[    C^{1}_{-\beta}(\real, \real^{2n}) = \ker(L) \oplus \mathcal{M}, \]
we can choose $\mathcal{M}$ to be $\ker(L)^{\perp}$, the orthogonal
complement with respect to the weighted inner product $ \left<u,v \right>_{\delta}$, which can be constructed due to the finite dimensionality of $\ker(L)$.
This being done, and having chosen an orthogonal basis $\{\dot{\gamma}(t),k_2(t),...,k_{2l-1}(t)\}$ for $\ker(L)$, we have
that $k+w(k)$ satisfies
\begin{align*} \int_{\real} e^{-2\delta\phi(t)}<\dot{\gamma}(t),(k+ w(k))(t)>dt =  0 &
                                     \Leftrightarrow  k \in \linspan\{k_2(t),...,k_{2l-1}(t)\} \end{align*}
since $w:\ker(L) \ra \ker(L)^{\perp}$.
We now show that all geometrically distinct homoclinics can be found by considering $\mathfrak{g}$ with the coefficient of $\dot{\gamma}(t)$ fixed at zero. We do this by proving:
\begin{prop} Every solution
\[ \tilde{\gamma}(t) = \gamma(t) + (k + w(k))(t) \]
with $k$ sufficiently small, can also be expressed as
\begin{equation}  \tilde{\gamma}(t) = \gamma(t+ \xi) + (k^* + w(k^*))(t+\xi) \label{timetranslation} \end{equation}
with  $k^* \in \linspan\{k_2(t),...,k_{2l-1}(t)\}$.
\end{prop}
In other words, the homoclinics obtained with nonzero coefficients of $\dot{\gamma}(t)$ are only time translations of those obtained with the coefficient of $\dot{\gamma}(t)$ set to zero. The following proof uses ideas from \cite{Knobloch1995}.
\begin{proof}
We apply the implicit function theorem to the functional
\[ P: C^1_{-\beta} \times \real \ra \real, \:\: P(x,\xi) :=  \int_{\real} e^{-2\delta\phi(t+\xi)}<x(t) - \gamma(t+\xi),\dot{\gamma}(t+\xi)>dt. \]
We observe that
\begin{enumerate}
\item $P(\gamma,0) = 0.$
\item $D_{\xi}P(x,\xi)|_{(\gamma,0)} = - \int_{\real} e^{-2\delta\phi(t)}<\dot{\gamma}(t),\dot{\gamma}(t)>dt \neq 0. $
\end{enumerate}
So we can apply the IFT and write
\[ P(x,\xi) = 0 \Leftrightarrow \xi = \xi^*(x) \]
for $(x,\xi)$ in a neighbourhood of $(\gamma,0)$. Now, since in the expression of our homoclinic $ \tilde{\gamma}(t) $,  $k$ is sufficiently small, we have that $\tilde{\gamma}$ is close to $\gamma$, and so we can write
\begin{align} 0 =   P(\tilde{\gamma},\xi^*(\tilde{\gamma})) = & \int_{\real} e^{-2\delta\phi(t+\xi^*(\tilde{\gamma}))}<\tilde{\gamma}(t) - \gamma(t+\xi^*(\tilde{\gamma})),\dot{\gamma}(t+\xi^*(\tilde{\gamma}))>dt \nn \\
 =  &  \int_{\real} e^{-2\delta\phi(t)}<\tilde{\gamma}(t-\xi^*(\tilde{\gamma})) - \gamma(t),\dot{\gamma}(t)>dt \label{kstar} \end{align}
So, the term $z^*(t) = \tilde{\gamma}(t-\xi^*(\tilde{\gamma})) - \gamma(t)$ is small, and
\[ \tilde{\gamma}(t) = \gamma(t+ \xi^*(\tilde{\gamma})) + z^*(t+\xi^*(\tilde{\gamma})) \]
so that $z^* = k^* + w(k^*)$, and by (\ref{kstar}) we have $k^* \in \linspan\{k_2(t),...k_{2l-1}(t)\}$. Hence, we have found the $k^*$
from equation (\ref{timetranslation}), so the claim is proved.
\end{proof}

\begin{lemma} \label{Hessianintegrals}
 For $i,j \in \{2,...,2l+1 \}$, we have
\begin{align*}  \frac{\p^2 \mathfrak{g}}{\p k_i \p k_j} & = \int_{\real} \left< \dot{\gamma}(t),  d_{x}^3 H(\gamma(t))(k_i(t),k_j(t)\right>dt \\ & = \int_{\real}  \frac{d}{dt} \left<k_i(t),  d_{x}^2H(\gamma(t))(k_j) \right> dt
\end{align*}
\end{lemma}
\begin{proof}
We observe that the integrand here can be written as
\[
 \left< \dot{\gamma}(t),  d_{x}^3 H(\gamma(t))(k_i(t),k_j(t))\right>  = \frac{d}{dt} \left<k_i(t),  d_{x}^2H(\gamma(t))(k_j) \right> \] \[ - \left< \dot{k_i}(t) , d_{x}^2 H(\gamma(t))(k_j(t))\right>   - \left< k_i(t), d_{x}^2 H(\gamma(t))( \dot{k_j}(t))\right>\]
But two of the terms on the right hand side here cancel out;
\begin{align*}    \left<\dot{k_i}(t), d_{x}^2 H(\gamma(t))(k_j(t))\right> & = -\left<\Jay d_{x}^2 H(\gamma(t))(k_i(t)), d_{x}^2 H(\gamma(t))(k_j(t)) \right> \\
& =  \omega(d_{x}^2 H(\gamma(t))(k_i(t)),d_{x}^2 H(\gamma(t))(k_j(t)) )
\end{align*}
and, since $d_{x}^2 H(\gamma(t))$ is symmetric,
\begin{align*}
\left< k_i(t), d_{x}^2 H(\gamma(t))( \dot{k_j}(t))\right> &  =  \left<\dot{k_j}(t), d_{x}^2 H(\gamma(t))(k_i(t))\right> \\
                                                                                                & =   \; \omega(d_{x}^2 H(\gamma(t))(k_j(t)),d_{x}^2 H(\gamma(t))(k_i(t)) )  \\
                                                                                                & =   - \omega(d_{x}^2 H(\gamma(t))(k_i(t)),d_{x}^2 H(\gamma(t))(k_j(t)) )
\end{align*}

since the symplectic form is skew-symmetric. Note that when $i=j$, both terms are zero.
\end{proof}
\begin{rem}
See also \cite{Blazquez-Sanz}, where similar calculations are
performed in a different bifurcation scenario.
\end{rem}

\section{The scattering matrix} \label{scatteringmatrix}

In order to evaluate the integrals from lemma \ref{Hessianintegrals} which define the elements of the Hessian matrix, we introduce the scattering matrix. This is a linear map defined on the centre subspace of the equilibrium which maps asymptotic initial conditions of the linearised variational equation from this sympectic subspace at negative infinity to their resting places in the same subspace at positive infinity, while accounting for the effects of the asymptotic motion in the center subspace. Since this map is defined using the (linear) Hamiltonian flow, and the space on which it is defined is symplectic, it is represented by a symplectic matrix. It is referred to as the \emph{scattering matrix}, and we call it $\sigma$. See also \cite{Lerman1996}, \cite{Yagasaki2000} and \cite{Yagasaki}. \\

Each $k(t) \in \linspan\{k_2, ... , k_{2l+1}\}$ approaches the orbit of a point in the center subspace as $t\ra\pm\infty$;
\[ \lim_{t\ra \pm\infty}k(t) = \Psi(t) k_{\pm\infty} \mbox{ with } k_{\pm\infty} \in E^c\]
with $\Psi(\cdot)$ denoting the fundamental matrix of the linear system on the center subspace $\dot{u} = -\Jay D^2H(0)|_{E^C}u(t)$. There is thus a family of $2l$-dimensional symplectic subspaces $Y(t) \subset T_{\gamma(t)}\real^{2n}$ $t\in\real$ spanned by the initial conditions $k_t$ such that $\Phi(s,t)k_t$ lies asymptotically in the center subspace $E^c$ at the equilibrium as $s \ra \pm \infty$. Let $\Phi^c(t,s) : Y(t) \ra Y(s)$, $s,t \in \real$ denote the restriction of the solution operator for the variational equation to these subspaces. Observing then that we can relate $k_{-\infty}$ to $k_{+\infty}$ via
\[ k_{+\infty} \: =  \bigg(  \lim_{t\ra\infty}\Psi(-t)\Phi(t,0)\bigg) \left(\lim_{t\ra -\infty} \Psi(-t)\Phi(t,0)\right)^{-1}k_{-\infty} \]
 we note that each of the limits in this definition exist:
\begin{prop}
The limits \[\lim_{t\ra\pm\infty} \Psi(-t)\Phi^c(t,0) \] exist and are nonsingular.
\end{prop}
\begin{proof}
We write
\begin{align*} \dot{y} & =  DX_H(0)|_{E^c}y(t) + \left( DX_H(\gamma(t))|_{Y^c(t)} - DX_H(0)|_{E^c} \right) y(t) \\
                                     & =:   DX_H(0)|_{E^c}y(t) + M(t)y(t) \end{align*}
noting that $M(t) =O(e^{-\lambda t})$ for $0<\lambda<\alpha$ (with $\alpha$ being the minimum of the real parts of the hyperbolic eigenvalues of the linearisation at the origin) as a consequence of the exponential convergence of the homoclinic orbit $\gamma(t)$ to the origin. We find solutions $\tilde{\phi}_j(t)$ such that
\[ \lim_{t\ra\infty} \tilde{\phi}_j(t) e^{-\lambda_j t} = p_j \]
where $DX_H(0)|_{E^c}p_j = \lambda_j p_j$ for each $p_j$. The $\tilde{\phi}_j(t)$ are found as fixed points of an operator $T_{t^*, j}$ mapping from the space of bounded continuous functions on the interval $[t^*,\infty)$, $ C([t^*,\infty),\real^{2l})$ with the supremum norm $|\cdot|_{\infty}$, into itself. We show that for $t^*$ sufficiently large, each $T_{t^*,j}$ is a contraction. The  $T_{t^*,j}$ are defined by
\[ T_{t*,j}(\phi(t)) = e^{\lambda_j t}p_j - \int_{t}^{\infty} e^{DX_H(0)|_{E^c}(t-s)}M(s)\phi(s)ds \]
We have:
\begin{align*} \| T_{t^*, j} \phi_1(t) - T_{t^*,j} \phi_2(t) \| & \leq |\phi_1(t) - \phi_2(t)|_\infty \int_{t}^{\infty} \| e^{DX_H(0)|_{E^c}(t-s)}\| Ce^{-\lambda s}ds \\
 & \leq  |\phi_1(t) - \phi_2(t)|_\infty CC_1 \frac{e^{-\lambda t^*}}{\lambda} \end{align*}
so this is a contraction for $t^*$ large enough, for each $j \in \{1,...,2l\}$. Using this approach for each $j$, we can build a fundamental matrix $\tilde{\Phi}(t) = \left(  \tilde{\phi}_1(t) \bigg| \: ... \:\bigg|  \tilde{\phi}_{2l}(t) \right)$ (that is, using the  $\tilde{\phi}_j(t)$ as columns), so that
\[ \lim_{t\ra\infty} \tilde{\Phi}(t) = \left( e^{t DX_H(0)|_{E^c}}p_1 \bigg| \: ... \: \bigg|  e^{t DX_H(0)|_{E^c}}p_{2l} \right) \]
which implies
\[ \lim_{t\ra\infty} \Psi(-t)\tilde{\Phi}(t) = P, \]
where $\det(P) \neq 0$. Now we can return to our original fundamental via $\Phi(t,0) = \tilde{\Phi}(t)\tilde{P}$ for a nonsingular matrix $\tilde{P}$. We conclude
\[ \lim_{t \ra \infty} \Psi(-t)\Phi(t,0) = P\tilde{P} \]
which is nonsingular. A similar argument holds in negative time.
\end{proof}

We then define the scattering matrix $\sigma : E^c \ra E^c$ by
\begin{equation} \sigma := \lim_{t\ra\infty} \Psi(-t)\Phi^c(t,-t)\Psi(-t). \label{sigmadefinition} \end{equation}
Thus, since $\Psi(t)$ is orthogonal and commutes with $D^2H(0)$, we have
\begin{align*} &   \lim_{t\ra +\infty} \left< D^2H(\gamma(t))k_i(t), k_j(t) \right> -  \lim_{t\ra-\infty} \left< D^2H(\gamma(t)) k_i(t),  k_j(t) \right>\\
   & =   \left< D^2H(0)\Psi(t) k_{i,+\infty}, \Psi(t) k_{j,+\infty} \right>  -  \left< D^2 H(0) \Psi(t) k_{i,-\infty}, \Psi(t) k_{j,-\infty} \right>  \\
   & =  \left< D^2H(0) k_{i,+\infty}, k_{j,+\infty} \right> -  \left< D^2 H(0) k_{i,-\infty}, k_{j,-\infty} \right> \end{align*}

which, together with the expression (\ref{reduced derivs 2}) leads to the following representation of the Hessian, concluding the proof of  theorem \ref{LSreductionthm}:
\begin{equation} D^2\mathfrak{g} =  \sigma^T D^2 H(0)|_{E^c} \sigma -D^2 H(0)|_{E^c}. \label{sigmahessian} \end{equation}

\subsection{Indefiniteness of the Hessian}
In this subsection we prove part $(i)$ of theorem \ref{mainthm}. The argument uses the classical minimax principle (see \cite{CHbook}), which states that
given a symmetric $(n\times n)$ matrix $A$ with the eigenvalues $\lambda_i$ ordered so that $\lambda_i\leq \lambda_{i+1}$, $i=1,\dots,n-1$,
$$\min_{\mathcal{R}} \max_{\|v\|=1,v\in \mathcal{R}} \langle v, Av \rangle = \lambda_k$$
where $\mathcal R$ runs all $(n+1-k)$-dimensional linear subspaces.
Combining it with the linear nonsqueezing theorem \cite{Mcduff1998}, we show that the most negative and most positive eigenvalues of $\sigma^T D^2H(0)|_{E^c}\sigma$ cannot be closer to zero than those of $D^2H(0)|_{E^c}$, which implies that $D^2\mathfrak{g}$ must be indefinite. Hence, if it is invertible, it can't have the signature $(0,2l)$ or $(2l,0)$. Recall, we assume (without loss of generality) that the matrix $D^2H(0)|_{E^c}$ takes the form  $D^2H(0)|_{E^c} = \diag(\omega_1,...,\omega_l,\omega_1,...,\omega_l)$.

\begin{proof}[Proof of theorem \ref{mainthm} $(i)$]
Seeking a contradiction, we assume that $G=D^2\mathfrak{g}$ is positive definite. This implies that the eigenvalues $\lambda_i$ of the symmetric matrix $\sigma^T D^2H(0)\sigma$ (ordered in increasing size) are larger than those of $D^2H(0)$. That is, they satisfy\footnote{This fact itself can also be proved using the minimax principle}
\[ \begin{cases}  \lambda_1, \lambda_2 & > \omega_1 \\... & \\ ... & \\ \lambda_{2l-1}, \lambda_{2l} & > \omega_l \end{cases} \]
We now consider the minimax principle for the first eigenvalue $\lambda_1$ of $\sigma^T D^2 H(0) \sigma$, which states;
\begin{equation} \label{minimax} \omega_1 < \lambda_1 = \min \{ \max \left< D^2H(0)\sigma v,\sigma v \right>  | \|v\| = 1, v \in U, \mbox{ $U$ subspace with } \dim(U) = 1 \}. \end{equation}
The $2$-dimensional symplectic eigenspace of $D^2H(0)$ associated with $\omega_1$ is $E_{\omega_1} = \linspan\{ q_j,p_j\}$ for some $j \in \{1,...,l\}$. Consider now the symplectic subspace $\sigma^{-1}(E_{\omega_1})$. By the linear version of Gromov's nonsqueezing theorem (see e.g.\cite{Mcduff1998}), the unit ball in $\real^{2l}$ cannot be mapped into the cylinder $C_r(q_j,p_j) = \{ (q,p) | q_j^2 + p_j^2 \leq r^2 \}$ for $r^2 < 1$, so either $\| \sigma v \| = 1$ for all $v \in \{\sigma^{-1}(E_{\omega_1})| \|v\| = 1\}$,
or there exist $v_+, v_- \in \{\sigma^{-1}(E_{\omega_1})| \|v\| = 1\}$ such that
\[ \| \sigma v_+ \| > 1,  \:\: \| \sigma v_- \| < 1. \]
In either case, we arrive at a contradiction to the statement (\ref{minimax}) of the minimax principle: in the former we can take any $v$ from $\{\sigma^{-1}(E_{\omega_1})| \|v\| = 1\}$ to get $\lambda_1 = \omega_1$, and in the latter we can take $v_-$ if $\omega_1 > 0$ or $v_+$ if $\omega_1 < 0$ to arrive at $\lambda_1 < \omega_1$.

If we assume instead that $G$ is negative definite, we can consider the minimax principle for the largest eigenvalue $\lambda_{2l}$, which in this case will give
\begin{equation} \label{maximin} \omega_l > \lambda_{2l} = \max \{ \min \left< D^2H(0)\sigma v,\sigma v \right>  | \|v\| = 1, v \in U, \mbox{ $U$ subspace with } \dim(U) = 1 \}. \end{equation}
A similar argument to the one above then yields $\omega_l \leq \lambda_{2l}$, the required contradiction.
\end{proof}

In the case of the smallest eigenvalue, the `max' in the minimax principle is redundant (likewise for the `min' for the largest eigenvalue). For other eigenvalues however, these elements come into play, meaning that in general the argument cannot be repeated to rule out other signatures.

\section{Near-integrable systems, and near-identity scattering matrices} \label{Section:Nearidentity}
 \begin{defn} A symplectic rotation is a real symplectic matrix $R_{\theta} = [r_{i,j}] \in Sp(2n,\real)$ with $\theta = (\theta_1,...,\theta_n) \in \real^n$ such that for each $i \in \{1,...,n\}$,
\[ \left( \begin{array}{cc} r_{i,i} & r_{i,n+i} \\ r_{n+i,i} & r_{n+i,n+i} \end{array} \right) =
\left( \begin{array}{cc} \cos \theta_i & \sin \theta_i \\ -\sin \theta_i & \cos \theta_i \end{array} \right)\]
 and $r_{i,j} = 0$ otherwise.
\end{defn}
So $R_\theta$ acts by a rotation through an angle $\theta_i$ in each pair of conjugate directions $(x_i, x_{n+i})$.

\begin{rem} For our considerations, the scattering matrix $\sigma$ is only determined up to left multiplication by a symplectic rotation,
since
\[ \left< D^2H(0)|_{E^c} R_{\theta} \sigma k_l , R_{\theta} \sigma k_m \right> = \left<R_{\theta} D^2H(0)|_{E^c}\sigma k_l , R_{\theta}  \sigma k_m \right>  =  \left< D^2H(0)|_{E^c}\sigma k_l ,  \sigma k_m \right>\]
Hence, considering the form (\ref{sigmahessian}) of the Hessian of our reduced function $\mathfrak g$,
we see that two scattering matrices $\sigma$ and $R_{\theta}\sigma$ are equivalent in the sense that they yield the same Hessian matrix.
\end{rem}

It is easy to build an example of an integrable system with a homoclinic loop. Consider a Hamiltonian $H_0$  of the form
\begin{align} H_0(q_1,...,q_l,p_1,...,p_l,x_1,...,x_{n-l},y_1,...,y_{n-l}) & = \\ h_c(q_1,...,q_l, p_1,...,p_l) \nn &
& \hspace{-1cm}+ h_s(x_1,...,x_{n-l},y_1,...,y_{n-l}), \end{align}
where the quadratic part of $h_c$ is $h_{c,2} = \sum_{i+1}^{l}\frac{\omega_i}{2}(q_i^2 + p_i^2)$, with each $\omega_i \in \mathbb{R}$ distinct. Let the $(n-l)$ degree of freedom Hamiltonian vector field given by $h_s$ have a hyperbolic equilibrium at the origin with a nondegenerate homoclinic orbit $\gamma_0(t)$, that is, a homoclinic along which the intersection of the tangent spaces to the stable and  unstable manifolds is one-dimensional. The orbit $\gamma_0(t)$ is a homoclinic loop of the system $X_{H_0}$, contained in the subspace $ \{ (\bf{q},\bf{p}) = 0 \} $. It is straightforward to see from the product structure of the system and the nondegeneracy of $\gamma_0(t)$ that the vector field $X_{H_0}$ satisfies the transversality assumption $\ref{transass}$ on the invariant manifolds. It is also straightforward to see from the product structure of the system and the nondegeneracy of $\gamma_0(t)$ that the vector field $X_{H_0}$ satisfies the transversality assumption $\ref{transass}$ on the invariant manifolds.

Note that the Hamiltonian $H_0$ can be chosen to be completely integrable. For instance, we could take
\begin{align} h_c(q_1,...,q_l,p_1,...,p_l) & =  \sum_{i=1}^{l}\frac{\omega_i}{2}(q_i^2 + p_i^2) \label{hc} \\
 h_s(x_1,...,x_{n-l},y_1,...,y_{n-l}) & = \frac{y_1^2}{2} - \frac{x_1^2}{2} + \frac{x_1^3}{3} + \sum_{i=2}^{n-l} \frac{\alpha_i}{2}(y_i^2 - x_i^2) \label{hs} \end{align}
with $\alpha_i \in \real$. This leads to a system which has a homoclinic loop in the $(x_1,y_1)$ plane given by $\gamma_0(t) = (g(t),\dot{g}(t))$, $g(t) = \frac{3}{2}\mbox{sech}^2(\frac{t}{2})$, $t\in\real$, and $n$ first integrals $H_0$, $\xi_1,...\xi_l$, $\eta_2,...\eta_{n-l}$ where
$\xi_i = \frac{\omega_i}{2}(q_i^2 + p_i^2)$ and $\eta_i =  \frac{\alpha_i}{2}(y_i^2 - x_i^2)$. These first integrals commute with respect to the standard Poisson bracket $\{f_1,f_2\}(\cdot) = \omega (X_{f_1}(\cdot),X_{f_2}(\cdot))$.
\begin{prop}
The scattering matrix of the orbit $\gamma_0(t)$ in the system given by $X_{H_0}$ as defined above, is the identity.
\end{prop}

\begin{proof}

In our case the variational equation along $\gamma_0(t)$ takes the form

\[ \begin{pmatrix} \dot{\bf{q}} \\ \dot{\bf{p}} \end{pmatrix} = \Jay_{l} D^2 h_{c}(0) \begin{pmatrix} \bf{q} \\ \bf{p} \end{pmatrix} \:\: , \:\:
\begin{pmatrix} \dot{\bf{x}} \\ \dot{\bf{y}} \end{pmatrix} = \Jay_{n-l} D^2 h_{s}(\gamma_0(t)) \begin{pmatrix} \bf{x} \\ \bf{y} \end{pmatrix}. \]
The $2l$-dimensional $ (\bf{q},\bf{p})$ subsystem has constant coefficients and the fundamental matrix is a symplectic rotation, $R_{t\bf{\omega}}$ , where $\bf{\omega} = (\omega_1,...,\omega_l)$. The $ (\bf{x},\bf{y})$ subsystem has only one bounded solution on $\real$ (as a consequence of the nondegeneracy of $\gamma_0(t)$); it is given by $\dot{\gamma_0}(t)$. For the scattering matrix, we find $\lim_{t \ra \infty}R_{-t\bf{\omega}}R_{2t\bf{\omega}}R_{-t\bf{\omega}} = I$.
\end{proof}

In general, we do not expect that the scattering matrix of a completely integrable system is identity. However, the integrable flow preserves the value of integrals,
so the linear map defined by the scattering matrix on the center manifold must preserve the linearized actions, i.e. it has to be a symplectic rotation for an appropriate
choice of coordinate system on the tangent space to $W^c$ (use the action-angle variables for the linearized system on $W^c$ as the polar coordinates).\\

Now, let us consider what kind of scattering matrices can appear at a small perturbation of an integrable system. We, first, consider perturbations which are localised
near $\gamma_0(t)$ on a finite time interval $[-T,T]$, for some $T >0$. We write the scattering matrix as a composition of symplectic matrices which represent the linear flow in the central subspace $Y(t)$ on $ [-\infty, -T]$, $[-T,T]$ and $[T,\infty]$ respectively;
$$\sigma = \lim_{t\ra+\infty} \Psi(-t) \Phi^c(t,T) \circ \Phi^c(T,-T) \circ \lim_{t\ra+\infty} \Phi^c(-T, -t) \Psi(-t).$$

We can always choose symplectic coordinates in $Y(T)$ and $Y(-T)$ such that $\lim_{t\ra+\infty} \Phi^c(-T,-t) \Psi(-t)=id$ and $\Phi^c(T,-T)=id$, so that
\begin{equation} \sigma = \lim_{t\ra+\infty} \Psi(-t) \Phi^c(t, T).\label{1f}\end{equation}
When we add a perturbation localised strictly inside a neighbourhood of $\{\gamma_0(t), t\in[-T,T]\}$, this would result to a small perturbation to
$\Phi^c(T,-T)$ only. Thus, the scattering matrix for the perturbed system will take the form
 \begin{equation}  \sigma= \sigma_0 \Phi^c(T,-T) \label{productsigma} \end{equation}
where $\sigma_0$ is the scattering matrix for the unperturbed system. Let us show that the localised perturbation can be chosen in such a way that
$\Phi^c(-T,T)$ will become any given symplectic matrix close to identity.

Indeed, take a small affine cross-section $\Sigma$ through $\gamma_0(-T)$, such that it would contain the central subspace $Y(-T)$.
Let $\mathcal U$ be the union of all forward orbits of length $2T$
over all initial points in $\Sigma$. The set $\mathcal U$ is foliated by the level sets of $H$
which are smooth manifolds of codimension one, invariant with respect to the flow maps $\varphi_\tau$. The maps $\varphi_\tau$ are symplectic and preserve $H$, which means we can always introduce symplectic $C^k$-coordinates\footnote{$k$ is the smoothness of the system, so we assume here that the Hamiltonian $H$ is at least of class $C^{k+1}$}
$(E,t,z)$ in $\mathcal U$ such that $E$ is the value of the Hamiltonian $H$, the coordinate $t\in[-T,T]$ equals to the time it takes for the point to get back to $\Sigma$,
and $z$ stays constant along the orbits, i.e. $(E,t,z)=\varphi_t(E,0,z)$. One can check that the symplectic form in $\mathcal U$ is given by
$dE\bigwedge dt + dz\bigwedge \Jay dz$.

In these coordinates the map $\bar\varphi=\varphi_{2T}$ is identity, so the map $\Phi^c(T,-T)$ obtained by the restriction of the derivative of $\varphi_{2T}$
at the point $\gamma_0(-T)$ to the central subspace $Y(-T)$ is also an identity, as required for formula (\ref{1f}) to be true. Let $C$ be a
symmetric matrix, $\mu\in R^{2n-2}$ be small, and $\varepsilon$ run a small interval of $R^1$ around zero. Consider a family of perturbed Hamiltonians $H_{\varepsilon,C,\mu}$ defined as follows:
\begin{equation}\label{pertham}
H_{\varepsilon,C,\mu}=E+\frac{\varepsilon}{2}\xi(t)\eta(E,z) \langle z, C z\rangle - \xi(t)\eta(E,z)\langle \mu , z\rangle
\end{equation}
where $\xi$ is localised strictly inside the interval $[-T,T]$,
$\eta$ is localized in a small neighbourhood of zero, $\eta=1$ for all $(z,E)$ close enough to zero,
and $\int_{-T}^T \xi(t)dt=1$. At $(z,E)$ close to zero the equations of motion by the perturbed Hamiltonian are given by
$$\dot t=1, \qquad \dot z= -\varepsilon \xi(t) \Jay C z + \mu.$$
These equations immediately imply
\begin{equation}\label{permap}
\bar\varphi_{\varepsilon,C}(0)=\mu(1+\mathcal{O}(\varepsilon)), \qquad \frac{\partial \bar\varphi_{\varepsilon,C}}{\partial x}(0)=\exp(-\varepsilon \Jay C).
\end{equation}
As we see, the solution $z=0$ of the unpertubed system that corresponds to the homoclinic loop $\gamma_0(t)$ persists in the perturbed system if $\mu=0$, i.e. the homoclinic
loop persists. Thus, if we denote by $B$ the restriction of the matrix $C$ to the center subspace $Y(-T)$, then
by (\ref{permap}),(\ref{productsigma}), we obtain the following result.

\begin{prop} \label{proppert}
Given any symmetric matrix $B$ there exists a family of perturbed Hamiltonian $H_{\varepsilon,C}$ such that $H_{0,C}\equiv H_0$, and for all small
$\varepsilon$ the system defined by $H_{\varepsilon,C}$ has a homoclinic loop to the elliptic-hyperbolic equilibrium $O$ with the corresponding scattering matrix $\sigma_{\varepsilon,C}$ equal to
\begin{equation}\label{sprf}
\sigma_{\varepsilon,C}=\sigma_0\exp(-\varepsilon \Jay B).
\end{equation}
\end{prop}

Recall that any symplectic matrix which is sufficiently close to the identity can be expressed as $\exp(\Jay B)$ with a small symmetric matrix $B$ (see eg. \cite{Gosson2006}). Thus, we have shown that arbitrary symplectic perturbation of the scattering matrix can be achieved by a small perturbation of the Hamiltonian $H_0$.

So far, the perturbations we considered were localised in a bounded domain, so they were not-analytic (of class $C^k$ if the non-perturbed Hamiltonian is $C^{k+1}$). However,
the statement of Proposition (\ref{proppert}) carries over to the analytic case as well. To see this, if the original Hamiltonian $H_0$ is analytic, consider an analytic family of perturbations $\hat H_{\varepsilon,C,\mu}$ which is at least $C^4$-close to the family of $C^4$-smooth localised Hamiltonians $H_{\varepsilon,C,\mu}$ built above.
By (\ref{permap}), the splitting of the separatrix loop in the family $H_{\varepsilon,C,\mu}$ is controlled by the parameter $\mu$. The same is true for any smooth approximation of this family (as the stable and unstable manifolds of $O$ depend continuously on the system, i.e. $C^2$-small changes in the Hamiltonian lead to small changes in the position of the stable and unstable manifolds), e.g. for the family $\hat H_{\varepsilon,C,\mu}$. This means,
in particular, that we can find $\mu(\varepsilon,C)$ smoothly depending on $C$ and $\varepsilon$ such that for all $\varepsilon$ and $C$ under consideration
the system defined by the Hamiltonian $\hat H_{\varepsilon,C}=\hat H_{\varepsilon,C,\mu(\varepsilon,C)}$ will have a homoclinic loop close to $\gamma_0$, and this loop
will analytically depend on $\varepsilon$ and $C$. As the family $\hat H_{\varepsilon,C}$ is at least $C^4$-close to the family $H_{\varepsilon,C,0}$, the corresponding
family of scattering matrices $\hat\sigma_{\varepsilon,C}$ will be close (as a smooth family) to the family $\sigma_{\varepsilon,C}$ defined by (\ref{sprf}). As the range
of possible values for $\sigma_{\varepsilon,C}$ covers all symplectic matrices close to $\sigma_0$, the same holds true for $\hat\sigma_{\varepsilon,C}$. Thus,
Proposition \ref{proppert} holds true for a family of analytic perturbations if $H_0$ is analytic.\\

As $\sigma_0$ is a symplectic rotation for an integrable system, and multiplication of the scattering matrix to a symplectic rotation does not
change the Hessian matrix $G=D^2\mathfrak{g}$, we obtain from (\ref{sprf}) that all Hessians $G$ that correspond to all possible near identity scattering matrices
can be realised by a small analytic perturbation of any given integrable system with a non-degenerate homoclinic loop to an elliptic-hyperbolic equilibrium.

\section{All indefinite signatures are possible} \label{Section:Allindef}

In light of the previous section, we now investigate the case in which the scattering matrix is a near identity symplectic transformation, which can be expressed as the flow along a Hamiltonian vector field. This means that we can write
\begin{equation} \sigma = \exp(-\varepsilon \Jay B) = I - \varepsilon \Jay B + \mathcal{O}(\varepsilon^2) \label{nearidentity} \end{equation}
with $\varepsilon << 1$ and $B$ an arbitrary symmetric matrix. Substituting the form (\ref{nearidentity}) into the expression (\ref{sigmahessian}) yields;
\[  \frac{1}{\varepsilon}\frac{\p^2 \mathfrak{g}}{\p \beta_i \p \beta_j}  =  B \Jay D^2 H(0) -  D^2 H(0)\Jay B + \mathcal{O}(\varepsilon) \]
Since the eigenvalues of a matrix depend continuously on its entries, for sufficiently small $\varepsilon$, the Hessian of $\mathfrak{g}$ has the same signature as $B \Jay D^2 H(0) -  D^2 H(0)\Jay B$. Our goal then, is to determine the possible signatures of this matrix. As a first observation, a simple calculation tells us that the trace is zero, which rules out the possibility that the matrix could be sign-definite, in agreement with theorem \ref{mainthm} $(i)$. For a deeper investigation, we begin by defining the map
\begin{align*} \chi_A : \Sym (\real^{2n \times 2n}) & \ra \Sym (\real^{2n \times 2n}) \\
                                                      \chi_A (B) & =  B \Jay A - A \Jay B \end{align*}
Where $\Sym (\real^{2n \times 2n})$ denotes the symmetric $2n \times 2n$ matrices with real entries. We can now express the set of matrices that we are studying as $ \mathcal{R}(\chi_{D^2 H(0)|_{W^c}})$. To gain a characterisation of this range, we endow $ \Sym (\real^{2n \times 2n})$ with the inner product
\begin{equation} \left< M_1 , M_2 \right> = \tr(M_1 M_2) \label{traceinnerproduct} \end{equation}
that is, the inner product of $M_1$ and $M_2$ is the trace of their ordinary matrix product. This allows us to write $ \mathcal{R}(\chi_{D^2 H(0)|_{W^c}}) = \ker (\chi^{*}_{D^2 H(0)|_{W^c}})^{\perp}$, where both the adjoint and the orthogonal complement are taken with respect to (\ref{traceinnerproduct}). We calculate the adjoint as follows
\begin{align*} \tr(\chi_A (B) M)  = & \; \tr( (B\Jay A - A\Jay B)M) \\
                                 =  &  \;  \tr(B \Jay A M) - \tr(A\Jay B M)  \\
                                 =  &  \; \tr(B\Jay A M) - \tr(B M A \Jay)    \\
                                 =  & \; \tr(B(\Jay AM - MA\Jay))
\end{align*}
Hence,
\[ \chi^{*}_{A}(M) =  \Jay AM - MA\Jay\]

Recall that in our coordinates the second derivative of the Hamiltonian restricted to the center subspace takes the diagonal form
\[ D^2 H(0)|_{E^c}  =  \diag(\omega_1, \ldots , \omega_n, \omega_1, \ldots , \omega_n ) \]
We also use the notation $\diag(M)$, for $M \in \real^{2n\times 2n}$, to denote the vector which contains the diagonal elements of $M$.

\begin{lemma} \label{Rangeofchi} If $\omega^2_1 \neq \omega^2_2 \neq \ldots \neq \omega^2_n$, then
\begin{enumerate}[(i)]
\item $ \ker ( \chi^{*}_{D^2 H(0)} ) =  \left\{ \left. \diag(a_{1},\ldots,a_{n},a_{1},\ldots,a_{n}) \right| a_{i} \in \real \right\} $
\item  $  \mathcal{R}(\chi_{D^2 H(0)}) =  \left\{ \left. M \in   \Sym (\real^{2n \times 2n}) \right|  \diag(M) =  (g_1, \ldots , g_n, -g_1, \ldots , -g_n), \;\; g_i \in \real \right\}  $
\end{enumerate}

\end{lemma}
\begin{proof}
\begin{enumerate}[(i)]
\item We use an induction argument on $n$. The statement is easily verified for $n=1$. Assuming the case $n=i$, we now consider $n=i+1$. We write $K \in \Sym (\real^{2(i+1)\times 2(i+1)})$ as
\[\hspace{-1cm} K =
\left( \begin{matrix} & &    k^1_{1, i+1} &  & &    k^2_{1, i+1}\\
& \vspace{2mm} \hspace{-2mm} K^1 &   \vdots & &  K^2 &   \vdots \\
 k^1_{i+1,1} & \ldots  & k^1_{i+1, i+1} &  k^2_{i+1,1} & \ldots  & k^2_{i+1, i+1} \\
 & &    k^2_{i+1,1} &  & &    k^3_{1, i+1}\\
& \vspace{2mm} \hspace{-2mm}  (K^2)^T &   \vdots & &  K^3 &   \vdots \\
 k^2_{1,i+1} & \ldots  & k^2_{i+1, i+1} &  k^3_{i+1,1} & \ldots  & k^3_{i+1, i+1}  \end{matrix}\right)
\]
with each $K^j$ being an $i \times i$ matrix, with $K^1$ and $K^3$ symmetric, and also writing $A^i$ for the matrix $\diag(\omega_1,\ldots,\omega_i)$, we arrive at
\[ K D^2H(0)|_{W^c} \Jay = \]
\[ \left( \begin{matrix} & &   -\omega_{i+1} k^2_{1, i+1} &  & &  \omega_{i+1}  k^1_{1, i+1}\\
& \vspace{2mm} \hspace{-2mm} - K^2 A^i &   \vdots & &  K^1 A^i &   \vdots \\
-\omega_1 k^2_{i+1,1} & \ldots  & -\omega_{i+1} k^2_{i+1, i+1} & \omega_1 k^1_{i+1,1} & \ldots  & \omega_{i+1} k^1_{i+1, i+1} \\
 & & -\omega_{i+1} k^3_{1,i+1} &  & &   \omega_{i+1} k^2_{i+1, i}\\
& \vspace{2mm} \hspace{-2mm} - K^3 A^i &   \vdots & &  (K^2)^T A^i &   \vdots \\
 -\omega_1 k^3_{1,i+1} & \ldots  & -\omega_{i+1} k^3_{i+1, i+1} &  \omega_1 k^2_{1,i+1} & \ldots  & \omega_{i+1} k^2_{i+1, i+1}  \end{matrix}\right) \]
and
\[ \hspace{-1cm}
\Jay D^2H(0)|_{W^c} K = \]
\[ \hspace{-4mm} \left( \begin{matrix} & &    \omega_1 k^2_{i+1, 1} &  & & \omega k^3_{1, i+1}\\
& \vspace{2mm} \hspace{-2mm} A^i (K^2)^T &   \vdots & &  A^i K^3 &   \vdots \\
 \omega_{i+1} k^2_{1,i+1} & \ldots  & \omega_{i+1}k^2_{i+1, i+1} &  \omega_{i+1} k^3_{1,i+1} & \ldots  & \omega_{i+1} k^3_{i+1, i+1} \\
 & &  -\omega_1  k^1_{1,i+1} &  & &  -\omega_1  k^2_{1, i+1}\\
& \vspace{2mm} \hspace{-2mm}  - A^i K^1 &   \vdots & & - A^i K^2 &   \vdots \\
-\omega_{i+1} k^1_{i+1,1} & \ldots  & -\omega_{i+1}k^1_{i+1, i+1} &  -\omega_{i+1}k^2_{i+1,1} & \ldots  & -\omega_{i+1} k^2_{i+1, i+1}  \end{matrix}\right) \]
Equating these matrices, we find that the block components are equal if and only if the matrix
\[ \left( \begin{matrix} K^1 & K^2 \\
(K^2)^T & K^3 \end{matrix} \right) \in \Sym(\real^{2i \times 2i}) \]
lies in the kernel for the $i$-dimensional case. By the induction hypothesis, this matrix thus has the form given in $(i)$. Equating the remaining components gives firstly
\begin{align*} -\omega_{i+1} k^2_{i+1,i+1} & = \omega_{i+1} k^2_{i+1,i+1}  & \Rightarrow   k^2_{i+1,i+1}  = 0 \\
\omega_{i+1} k^1_{i+1,i+1} & = \omega_{i+1} k^3_{i+1,i+1} & \Rightarrow  k^1_{i+1,i+1}  = k^3_{i+1,i+1} \end{align*}
Furthermore, we obtain a collection of pairs of simultaneous linear equations, one example being
\[ \left( \begin{matrix} w_{i+1} & -\omega_1 \\
                                   -\omega_1 & \omega_{i+1} \end{matrix} \right) \left( \begin{matrix} k^1_{1,i+1} \\ k^3_{1,i+1} \end{matrix} \right) = \left( \begin{matrix} 0 \\ 0 \end{matrix} \right) \]
If $\omega^2_{i+1} - \omega^2_1 \neq 0$, we thus obtain that $ k^1_{1,i+1} = k^3_{1,i+1} = 0$. Accounting for all components in a similar way tells us that provided $ \omega^2_{i+1} \neq \omega^2_{k}$ for $k \in \{ 1, \ldots , i\}$, we must have all other components equal to zero. Thus the only degree of freedom is in choosing the value of $ k^1_{i+1,i+1}  = k^3_{i+1,i+1}$, and so $K$ itself is of the form given in $(i)$. This concludes the induction step and thus the proof of $(i)$.

\item This follows easily from $(i)$, using the characterisation $ \mathcal{R}(\chi_{D^2 H(0)|_{W^c}}) = \ker (\chi^{*}_{D^2 H(0)|_{W^c}})^{\perp}$.
\end{enumerate}
\end{proof}

In this section we prove the following theorem:

\begin{thm}
The first order approximation to the Hessian
\[   D^2 H(0)\Jay B - B \Jay D^2 H(0) \]
can take every signature except $(2l,0)$ or $(0,2l)$. \label{anysignature}
\end{thm}

The proof is based upon an application of a theorem from \cite{Mirsky1958}. Before stating the theorem we introduce some notation

\begin{defn}
For two vectors $(a_1, \ldots , a_n)$ and $(b_1, \ldots , b_n)$ in $\real^n$, the expression
\[ (a_1, \ldots , a_n) \prec (b_1, \ldots , b_n) \]
will mean that when the elements are renumbered so that
\[ a_1 \geq \ldots \geq a_n, \mbox{  and  } b_1 \geq \ldots \geq b_n, \]
then
\begin{align}
a_1 + \ldots + a_k & \leq b_1 + \ldots + b_k \mbox{     } (k = 1, \ldots \ n-1)  \label{partialsums}\\
a_1 + \ldots + a_n & = b_1 + \ldots + b_n.
\end{align}
\end{defn}

\noindent{\bf Mirsky Theorem} \cite{Mirsky1958}.{\em
Let $ \omega_1, \ldots , \omega_n$, $a_1, \ldots , a_n$ be real numbers. Then
\[ (a_1, \ldots , a_n) \prec (\omega_1, \ldots , \omega_n ) \label{orderrelation} \]
is the necessary and sufficient condition for the existence of a real symmetric $n \times n$ matrix with $ \omega_1, \ldots , \omega_n$
as its eigenvalues and  $a_1, \ldots , a_n$, in that order, as its diagonal elements.}\\

We now use this criterion to prove theorem \ref{anysignature}. The idea of the proof will be to demonstrate that taking the vector $g$ given by
\[ (g_1,\ldots , g_l, g_{l+1}, \ldots , g_{2l}) = (1,\ldots ,1, -1, \ldots , -1), \]
and any $m \in \{1,\ldots, 2l -1 \}$, we can demonstrate a vector $b \in \real^{2l}$ with $m$ positive and $(2l - m)$ negative elements, satisfying
\[ (g_1, \ldots , g_{2l}) \prec (b_1, \ldots , b_{2l}). \]
Appealing to Mirsky theorem will then provide us with a matrix in $ G \in \Sym (\real^{2l \times 2l})$ whose diagonal elements are given by $g$ (and hence $G \in \mathcal{R}(\chi_{D^2 H(0)})$), whose eigenvalues are $b_1, \ldots , b_{2l}$, and hence $G$ has signature $(m, 2l -m)$.

\begin{proof}[Proof of theorem \ref{anysignature}]
Choose  any $m \in \{1,\ldots , 2l-1 \}$, and write
\[ b = \Big( \underbrace{2l -m, 1, \ldots, 1}_{m \mbox{ elements}},\underbrace{\frac{-(2l-1)}{(2l-m)},\frac{-(2l-1)}{(2l-m)}, \ldots , \frac{-(2l-1)}{(2l-m)}}_{(2l-m) \mbox{ elements}} \Big) \]

As explained above, the theorem will be proved if we can demonstrate that $g \prec b$ (with $g$ as defined above). Firstly, we note that the elements of $g$ and $b$ are already numbered in the appropriate nonincreasing order, and that
\[ g_1 + \ldots + g_{2l}  = b_1 + \ldots + b_{2l} = 0. \]
To prove that (\ref{partialsums}) is satisfied, we consider the cases $m>l$ and $m\leq l$ separately.

\noindent \underline{Case (1a): $m>l$, $k \in \{1,\ldots,l\}$}.  For $k$ in this range, the inequalities in (\ref{partialsums}) take the form
\begin{align*}   k \; \leq  & \; (2l-m) + (k-1) \\
           \Leftrightarrow  0 \; \leq & \; 2l - m -1 \end{align*}
which is true since $m \in \{1, \ldots , 2l -1 \}$.

\noindent \underline{Case (1b): $m>l$, $k \in \{l+1, \ldots, m \}$ } Here (\ref{partialsums}) becomes
\[ 2l - k \; \leq  \; (2l-m) + (k-1) \]
so
\[ -k \; \leq  \;  (k-1) - m \]
and since $l+1 \leq k \leq m$, this means
\[ l+1 -m -1 \leq k - m - 1 \]
so we need $l-m \geq -k $. But $ m \leq (2l-1)$ so
\begin{align*}
l-m \geq & \;  l - (2l-1) \\
        \geq      & - l -1 \\
        \geq      &  - k.
 \end{align*}
\underline{Case (1c): $m>l$, $k \in \{ m+1, \ldots , 2l \}$ }. We now have
\[ 2l - k \leq (2l -1) - (k - m)\frac{(2l -1)}{(2l - m)} \]
and since $(2l-m) > 0$ this simplifies to
\[ km \leq k + 2l(m-1). \]
Assuming for contradiction that $km > k + 2l(m-1)$ leads to
\[ m -1 > \frac{2l}{k}(m-1) \]
but since $\frac{2l}{k} \geq 1$, this is our required contradiction.

\noindent \underline{Case (2a): $m\leq l$, $k\in \{1,\ldots, m \}$ } This is the same as case (1a).

\noindent \underline{Case (2b): $m \leq l$, $k \in \{m,\ldots , l \}$} We now need to show
\[ k \leq (2l -1) - (k-m)\frac{(2l-1)}{(2l-m)}. \]
This simplifies to
\[ \frac{k}{2l}(4l - m -1) \leq (2l-1) \]
and since $\frac{k}{2l} \leq \frac{1}{2}$ and $(4l -m -1) \leq (4l - 2)$, this is true.


\noindent \underline{Case (2c): $ m\leq l$, $ k \in \{l+1, \ldots , 2l \}$} This is the same as case (1c).
\end{proof}

This result finishes the proof of item (iii) of Theorem 2.

\section{Reversible Hamiltonian case} \label{specialcases}

Let us now assume further that our Hamiltonian system is reversible with respect to a linear involution which acts antisymplectically $ R: \real^{2n} \ra \real^{2n}$, and also that the homoclinic to the equilibrium $\gamma(t)$ is \emph{symmetric}, as described by assumption \ref{Reversibility}. This implies that $R$ and $DX_H(0)$ share the same invariant subspaces, and in particular the restriction of $\Jay DX_H(0)$ to the center subspace $E^c$ is reversible with the respect to the restriction of $R$ to $E^c$. By a symplectic change of coordinates in $E^c$ which amounts to averaging the inner product over the finite group generated by $R$ and $\Jay$, we are able to assume without loss of generality that $\Jay$ takes its standard form $\Jay = \left( \begin{smallmatrix} 0 & I_l \\ -I_l & 0 \end{smallmatrix} \right)$ and $R$ is orthogonal (see for instance appendix B of \cite{Hoveijn2003}). Since $R^2 = I$, this means that $R$ is symmetric. In what follows we sometimes write $R$ for the restriction of $R$ to $E^c$, when the context is clear.

In this section we prove Theorem \ref{specialcasesthm}. First, we assemble some properties of the scattering matrix and the Hessian.

\begin{lemma}  Under assumption \ref{Reversibility},
\begin{enumerate}[(i)]
\item The scattering matrix $\sigma$ satisfies $\sigma \circ R \circ \sigma = R \label{Sigmarev}$.
\item $ D^2\mathfrak{g} \circ (R \circ \sigma) = -(R\circ \sigma)^T \circ D^2\mathfrak{g} . \label{Galmostrev} $
\end{enumerate}
\end{lemma}
\begin{proof} (i.) The scattering matrix is defined as $ \lim_{t\ra \infty} \Psi(-t)\Phi^c(t, -t)\Psi(-t) $.  As a consequence of assumption \ref{Reversibility} we have $R\Phi(t,-t) = \Phi(-t,t) R$ and since the dynamics in the centre subspace of the equilibrium are reversible, we also have $R\Psi(-t) = \Psi(t) R$. Furthermore, the family of subspaces $Y^c(t) \subset T_{\gamma(t)}\real^{2n}$ satisfy $RY^c(t) = Y^c(-t)$ which leads to $ P^c (t) = RP^c(-t)R$ where $P^c(t)$ is our projection onto $Y^c(t)$. Combining these relations and applying them to the definition of $\sigma$ yields $R\circ\sigma = \sigma^{-1} \circ R$ and hence the result.
Regarding part (iii), we already have the expression $-D^2\mathfrak{g} =  D^2H(0)|_{E^c} - \sigma^T (D^2H(0)|_{E^c})\sigma$, so that
\[ -D^2\mathfrak{g} \circ (R \circ \sigma) =   (D^2H(0)|_{E^c}) R \sigma - \sigma^T  (D^2H(0)|_{E^c}) \sigma R \sigma. \]
Since the linearisation $\Jay D^2 H(0)|_{E^c}$ is reversible, and since $R$ acts antisymplectically, this implies that $D^2H(0)$ commutes with $R$. Using this fact and (\ref{Sigmarev}) brings us to
\[ -D^2\mathfrak{g} \circ (R \circ \sigma) =  R(D^2H(0)|_{E^c}) \sigma - \sigma^T R (D^2H(0)|_{E^c}). \]
The claim now follows using $R^T = R$ and (\ref{Sigmarev}) again.
\end{proof}

The idea in what follows
is to choose a basis of $\ker(L)$ in which $D^2\mathfrak{g}$ becomes $R\circ \sigma$ reversible, thus implying a symmetry of the spectrum,
which gives the $(l,l)$ signature. Looking at (\ref{Galmostrev}), we see that $D^2\mathfrak{g}$ is $(R \circ \sigma)$ reversible if $(R \circ \sigma)$ is symmetric. Since $(R\circ \sigma)$ is an involution, this is the same as being orthogonal.

\begin{proof}[Proof of theorem \ref{specialcasesthm}]
Define a new inner product by
\begin{align*} [ x, y ]  = & \frac{1}{2}\left( \left< x,y \right> + \left< (R\circ \sigma) x, (R\circ \sigma) y \right> \right) \\
                   = & \left< \frac{1}{2} (I + (R \circ \sigma)^T (R \circ \sigma))x, y \right>. \end{align*}
Note that
\begin{equation} [ (R \circ \sigma)x, (R \circ \sigma)y ] = [ x, y ]. \label{invariance} \end{equation}
Since $\frac{1}{2}(I + (R \circ \sigma)^T(R \circ \sigma))$ is symmetric and positive definite, it has a uniquely defined symmetric
square root so we can write
\[ \frac{1}{2}(I + (R \circ \sigma)^T(R \circ \sigma)) = S^T S \]
and hence
\[ [x,y] = \left< Sx, Sy \right> \]
So, the new inner product is just the old one but in the new basis given by applying $S$ to the old basis. Looking at
(\ref{invariance}) tells us that in this basis, $R\circ \sigma$ is an isometry, and hence represented by an orthogonal
matrix. So, in this basis we have the relation
\[ D^2\mathfrak{g} \circ (R \circ \sigma) = - (R \circ \sigma) \circ D^2\mathfrak{g} \]
which is what we wanted, and so the signature of $D^2\mathfrak{g}$ must be $(l,l)$, since $D^2\mathfrak{g}$ is related to $-D^2\mathfrak{g}$ by a similarity transform.
\end{proof}

\section*{Acknowledments} This work was supported by the RSF grant 14-41-00044 at the Lobachevsky University of Nizhny Novgorod.

\end{document}